\documentclass[12pt]{amsart}
\usepackage{amssymb,amsmath,amscd}

\usepackage{amsthm}
\usepackage[T1]{fontenc}

\usepackage{mathrsfs}
\usepackage{color}
\usepackage{epsfig}

\usepackage[all]{xy}

\usepackage{mathtools}

\usepackage{appendix}

\theoremstyle{plain}

\newtheorem{theorem}{Theorem}[section]

\newtheorem{proposition}[theorem]{Proposition}
\newtheorem{corollary}[theorem]{Corollary}

\theoremstyle{definition}
\newtheorem{definition}[theorem]{Definition}
\newtheorem{notation}[theorem]{Notation}

\theoremstyle{remark}
\newtheorem{remark}[theorem]{Remark}

\newcommand{\ZZ}{\mathbb{Z}}
\newcommand{\QQ}{\mathbb{Q}}

\newcommand{\CC}{\mathbb{C}}

\newcommand{\HH}{\mathbb{H}}

\begin{document}

\title{Quasi-modularity of Generalized sum-of-divisors functions}

\author{Simon C. F. Rose}
\address{Max-Planck-Institut f\"ur Mathematik}
\email{simon@mpim-bonn.mpg.de}

\begin{abstract}
In 1919, P. A. MacMahon studied generating functions for generalized divisor sums. In this paper, we provide a framework in which to view these generating functions in terms of Jacobi forms, and prove that they are quasi-modular forms.
\end{abstract}

\maketitle

\section{Introduction}

In \cite{andrews_rose}, a relationship between certain generating functions studied by P. A. MacMahon \cite{macmahon} and Chebyshev polynomials is investigated. MacMahon introduces the following functions
\begin{gather*}
A_k(q) = \sum_{0 < m_1 < \cdots < m_k} \frac{q^{m_1 + \cdots + m_k}}{(1 - q^{m_1})^2 \cdots (1 - q^{m_k})^2}
\\
C_k(q) = \sum_{\substack{
0 < m_1 < \cdots < m_k \\
m_i \equiv 1 \pmod 2
}}
\frac{q^{m_1 + \cdots + m_k}}{(1 - q^{m_1})^2 \cdots (1 - q^{m_k})^2}
\end{gather*}
where we have a slightly different convention than used in \cite{macmahon} for the function \(C_k(q)\) (that is, we have written the summation in terms of congruences mod 2) which will be more suited to the results in this paper.

These are generating functions for a generalized sum-of-divisors functions in the following sense. If we write
\[
A_k(q) = \sum_{m=1}^\infty a_{m,k}q^m
\]
then the coefficients \(a_{m,k}\) are given by
\[
a_{m,k} = \sum s_1 \cdots s_k
\]
where the sum is taken over all ways of writing \(m = s_1m_1 + \cdots + s_km_k\) with the restriction that \(0 < m_1 < \cdots < m_k\). Similarly, if we write
\[
C_k(q) = \sum_{m=1}^\infty c_{m,k}q^m
\]
then the coefficients \(c_{m,k}\) are given by the sum
\[
c_{m,k} = \sum s_1 \cdots s_k
\]
where the sum is instead taken over all ways of writing \(m = s_1m_1 + \cdots + s_km_k\) with \(0 < m_1 < \cdots < m_k\) and where
\begin{equation}\label{eq_congruence}
m_i \equiv 1 \pmod 2.
\end{equation}

In \cite{andrews_rose} it is proven that these generating functions are related to Chebyshev polynomials in the following way.

\begin{definition}
We define the \(n\)-th Chebyshev polynomial (of the first kind) to be the unique polynomial \(T_n(x)\) such that
\[
T_n(\cos\theta) = \cos (n\theta).
\]
The first few of these are given as follows.
\begin{align*}
T_1(x) &= x & T_2(x) &= 2x^2 - 1 \\
T_3(x) &= 4x^3 - 3x & T_4(x) &= 8x^4 - 8x^2 + 1
\end{align*}
\end{definition}

\begin{theorem}\label{thm_old}
We have the following equalities of two-variable generating functions
\begin{equation}\label{eq_rose_andrews_thm}
\begin{split}
2\sum_{n=0}^\infty T_{2n+1}(\tfrac{1}{2}x)q^{n + 1 \choose 2} = (q; q)^3_\infty \sum_{k=0}^\infty A_k(q)x^{2k+1}, \\
1 + 2 \sum_{n=1}^\infty T_{2n}(\tfrac{1}{2}x)q^{n^2} = (q^2;q^2)_\infty(q;q^2)_\infty^2 \sum_{k=0}^\infty C_k(q)x^{2k}
\end{split}
\end{equation}
where
\[
(a;q)_\infty = \prod_{m=0}^\infty (1 - aq^m)
\]
is the \(q\)-Pochhammer symbol.
\end{theorem}

\begin{remark}
It should be noted that this again uses a slightly different convention than in \cite{andrews_rose} which is more suited to this paper. This follows from the equality
\[
(q^2;q^2)_\infty(q;q^2)_\infty^2 = \frac{(q;q)_\infty}{(-q;q)_\infty}.
\]
\end{remark}

In MacMahon's paper there is a plethora of other functions studied which essentially arise from varying the congruence condition \eqref{eq_congruence}. The goal of this paper is to generalize the equalities \eqref{eq_rose_andrews_thm} to include MacMahon's other functions. In this more general setting, certain Jacobi forms play a central role.


We will now provide a definition of the functions of interest.

\begin{definition}
Fix an integer \(n\). We say that a set of representatives (of congruence classes mod \(n\)) \(S \subset \{1, \ldots, n\}\) is {\em symmetric} if, for all \(\ell \in S\) we also have that \(- \ell \in S \pmod n\). Note that if \(n \in S\), then this is vacuously true. Note further that we are considering elements of \(S\) to be integers first and foremost, and that all statements such as \(m \in S \pmod n\) simply mean that \(m\) is congruent to some element in \(S \pmod n\).

If \(n \notin S\), then we will write \(S = \{\ell_1, \ldots, \ell_s\}\) with \(s = |S|\). If \(n \in S\), then we will write \(S = \{n = \ell_0, \ell_1, \ldots, \ell_s\}\). If we do not differentiate between the two cases, we will simply index our representatives as \(\ell \in S\).
\end{definition}

\begin{definition}
Fix an integer \(n\), and choose a symmetric set of representatives \(S\). We then define
\[
A_{S, n, k}(q) = \sum_{\substack{0 < m_1 < \cdots < m_k \\ m_i \in S \pmod n}} \frac{q^{m_1 + \cdots + m_k}}{(1 - q^{m_1})^2 \cdots (1 - q^{m_k})^2}.
\]
\end{definition}

If we expand out \(A_{S, n, k}(q)\) as the \(q\)-series \(A_{S, n, k}(q) = \sum_{m = 0}^\infty a_{S,n,k,m}q^m\) then the coefficients \(a_{S, n, k, m}\) are given by
\[
a_{S, n, k, m} = \sum s_1 \cdots s_k
\]
where we sum over all ways of writing \(m = m_1s_1 + \cdots + m_ks_k\) with \(0 < m_1 < \cdots < m_k\) and with each \(m_i \in S \pmod n\).

In fact, we have the following other specializations (see \cite{macmahon} for the definitions of the generating functions \(E_k(q)\) and \(G_k(q)\)):
\begin{gather*}
A_{\{1\},1,k}(q) = A_k(q), \\
A_{\{1\}, 2, k}(q) = C_k(q), \\
A_{\{1, 4\}, 5, k}(q) = E_k(q), \\
A_{\{2,3\},5,k}(q) = G_k(q),
\end{gather*}
and we also have that
\begin{gather*}
A_{\{n\},n,k}(q) = A_k(q^n), \\
A_{\{1,\ldots, n\}, n, k}(q) = A_k(q).
\end{gather*}

We will also look at the following generating functions, which provide a generalization of the functions \(B_k(q), D_k(q), F_k(q)\), and \(H_k(q)\) of \cite{macmahon}. We will not focus so much on these functions throughout this paper, although most of our results hold similarly for this with the appropriate minor changes.

\begin{definition}
Fix an integer \(n\) and a symmetric set of representatives \(S\). Define
\[
B_{S, n, k}(q) = \sum_{\substack{0 < m_1 < \cdots < m_k \\ m_i \in S \pmod n}} \frac{q^{m_1 + \cdots + m_k}}{(1 + q^{m_1})^2 \cdots (1 + q^{m_k})^2}.
\]
\end{definition}

Lastly, we need the following notationally convenient function.

\begin{definition}
For \(1 \leq \ell \leq n\), define
\[
\alpha(n, \ell) = \frac{\ell}{n} - \frac{1}{2}.
\]
\end{definition}

Note in particular that \(\alpha(n, \ell) = - \alpha(n, n - \ell)\), which relates to the requirement that a set of representatives \(S\) be symmetric.

Finally, let us define the following functions which will be of use.

\begin{definition}\label{def_theta}
Let \(r \in \QQ\), and let \(t\) be a non-negative integer. Then we define
\[
\vartheta_r(q, z) = \sum_{m \in \ZZ + r} q^{\frac{1}{2}m^2} z^m
\]
(where we consider \(q = e^{2\pi i \tau}\) and \(z = e^{2\pi i \sigma}\)). Furthermore, we define
\[
\vartheta_r^{(t)}(q) = \Big[\Big(z\frac{\partial}{\partial z}\Big)^t \vartheta\Big](q, -1) = \sum_{m \in \ZZ + r} e^{\pi i m} m^t q^{\frac{1}{2}m^2}
\]
and
\[
\widetilde{\vartheta}_r^{(t)}(q) = \Big[\Big(z\frac{\partial}{\partial z}\Big)^t \vartheta\Big](q, 1) = \sum_{m \in \ZZ + r} m^t q^{\frac{1}{2}m^2}.
\]
If \(t = 0\), then we omit it from the notation.

Finally, we will define the Dedekind \(\eta\)-function to be
\[
\eta(q) = q^{\frac{1}{24}}\prod_{m=1}^\infty (1 - q^m).
\]
\end{definition}

\begin{theorem}\label{thm_main_A}
Fix a non-negative integer \(n\) and a symmetric set of representatives \(S\). Then we have that
\begin{equation}\label{eq_thm_main_A}
\begin{aligned}
\sum_{k=0}^\infty (-1)^k A_{S, n, k}(q) x^{2k} &= \prod_{j=1}^s\frac{\vartheta_{\alpha(n, \ell_j)}(q^n, -z)}{\vartheta_{\alpha(n, \ell_j)}(q^n)} & n \notin S \\
\sum_{k=0}^\infty (-1)^k A_{S, n, k}(q) x^{2k + 1} &= -i\frac{\prod_{j=0}^s\vartheta_{\alpha(n, \ell_j)}(q^n, -z)}{\eta(q^n)^3\prod_{j=1}^s\vartheta_{\alpha(n, \ell_j)}(q^n)} &n \in S
\end{aligned}
\end{equation}
where \(x = z^{1/2} - z^{-1/2}\).
\end{theorem}

\begin{theorem}\label{thm_main_B}
Fix a non-negative integer \(n\) and a symmetric set of representatives \(S\). Then we have that
\[
\begin{aligned}
\sum_{k=0}^\infty B_{S, n, k}(q) x^{2k } &= \prod_{j=1}^s\frac{\vartheta_{\alpha(n, \ell_j)}(q^n, z )}{\widetilde{\vartheta}_{\alpha(n, \ell_j)}(q^n)} & n \notin S
\\
\sqrt{1 + \big(\tfrac{x}{2}\big)^2}\sum_{k=0}^\infty B_{S, n, k}(q) x^{2k} &= \frac{\prod_{j=0}^s\vartheta_{\alpha(n, \ell_j)}(q^n, z )}{\prod_{j=0}^s\widetilde{\vartheta}_{\alpha(n, \ell_j)}(q^n)} & n \in S
\end{aligned}
\]
where again \(x = z^{1/2} - z^{-1/2}\).
\end{theorem}

The connection between Theorems \ref{thm_main_A} and \ref{thm_main_B} and Theorem \ref{thm_old} is that the \(\vartheta\)-functions can themselves be written in terms of Chebyshev polynomials.

The main interest in the study of these functions is to understand their modular properties. The functions \(A_k(q)\) and \(C_k(q)\) are proven to be quasi-modular in \cite{andrews_rose}, and it is natural to ask whether or not the same holds in this generalized setting. We obtain the following.

\begin{theorem}\label{thm_top_weight}
Assume that \(n \notin S\). The functions \(A_{S, n, k}(q)\) are quasi-modular forms of weight at most \(2k\) for some congruence subgroup \(\Gamma\) of \(SL_2(\ZZ)\). The pure weight parts depend only on \(S\) (and \(n\)). In particular, the weight \(2w\) part is given by a multiple of
\[
\sum_{i_1 + \cdots + i_s = 2w} {2w \choose {i_1, \ldots, i_s}} \prod_{\ell \in S} \frac{\vartheta_{\alpha(n, \ell)}^{ (i_\ell)}(q^n)}{\vartheta_{\alpha(n, \ell)}(q^n)}.
\]
\end{theorem}

\begin{theorem}\label{thm_top_weight_B}
Assume that \(n \in S\), and write \(S = \{n, \ell_1, \ldots, \ell_s\}\). Then the functions \(A_{S, n, k}(q)\) are quasi-modular forms of weight at most \(2k\) for some congruence subgroup \(\Gamma\) of \(SL_2(\ZZ)\). The pure weight parts depend only on \(S\) (and on \(n\)). In particular, the weight \(2w\) part is given by some multiple of
\[
 \sum_{(2i_0+1) + i_1 + \cdots + i_s = 2w+1} {2w + 1 \choose 2i_0 + 1, i_1, \ldots, i_s} \Big(\frac{2}{n}\Big)^{i_0}\frac{\Big[\Big(q\frac{d}{dq}\Big)^{i_0} \eta(q^n)^3 \Big]\prod_{j=1}^s \vartheta_{\alpha(n, \ell_j)}^{(i_j)}(q^n)}{\eta(q^n)^3 \prod_{j=1}^s \vartheta_{\alpha(n, \ell_j)}(q^n)}.
\]
\end{theorem}

\begin{remark}
We can computationally verify that the group in question for both of these theorems appears to be \(\Gamma_1(n)\). In fact, this has recently been proven in \cite{larson}.
\end{remark}

\begin{remark}
For the case where these overlap the paper \cite{andrews_rose}, this is actually a strengthening of that result; in that paper, all that is given is that the generating functions \(A_k(q), C_k(q)\) are quasi-modular forms of weight at most \(2k\), with no statement about the constancy of the pure-weight pieces.

In particular, in that case the quasi-modularity was given by showing that the functions \(A_k(q), C_k(q)\) can be defined recursively using derivatives of \(A_{k'}(q)\) and \(C_{k'}(q)\) (with \(k' < k\)), which imply quasi-modularity; this is in contrast to the case here where we prove quasi-modularity by writing an explicit description of each weight \(2w\) piece (up to multiplication by a scalar).
\end{remark}

\begin{notation}
Throughout this paper, we will use the notation \(D_\xi\) for the differential operator \(\xi \frac{\partial}{\partial \xi}\).
\end{notation}

\subsection*{Acknowledgements}

The author would like to thank the Max-Planck-Institut f\"ur Mathematik for hosting him while this research was being prepared. Furthermore, he would also like to thank Kathrin Bringmann, Robert Osburn, Noriko Yui, and Don Zagier for fruitful discussions about this topic, as well as the referees for their helpful comments.

\section{Jacobi Forms and \(\vartheta\)-functions}\label{sec_jac_forms}

We will describe now the relevant properties of Jacobi forms. For a more thorough reference one can consult \cite{eichler_zagier}.

\begin{definition}\label{def_q_jacobi}
Fix half-integers \(k, m\). A {\em Jacobi form} of weight \(k\) and index \(m\) for a subgroup \(\Gamma\) of \(SL_2(\ZZ)\) is a holomorphic function \(\phi : \HH \times \CC \to \CC\) which satisfies, for \(\begin{pmatrix} a & b \\ c & d \end{pmatrix} \in \Gamma\)
\begin{equation}\label{eq_mod_transform}
(c\tau + d)^{-k} \exp\Big(\frac{-2\pi i m c \sigma^2}{c\tau + d}\Big) \phi\Big(\frac{a \tau + b}{c\tau + d}, \frac{\sigma}{c\tau + d}\Big) = \phi(\tau, \sigma),
\end{equation}
(the modular transformation). Furthermore, for \(r, s \in \ZZ\), it satisfies
\[
\exp\big(2\pi i m(r^2 \tau + 2 r \sigma )\big) \phi(\tau, \sigma + r \tau + s) = \phi(\tau, \sigma)
\]
(the elliptic transformation). 
\end{definition}

\begin{proposition}\label{prop_jac_qmod}
Let \(\phi(\tau, \sigma)\) be a Jacobi form of weight \(k\) for some group \(\Gamma\), and write
\[
\phi(\tau, \sigma) = \sum_{m=0}^\infty \phi_m(\tau) \sigma^m
\]
as the Taylor expansion of \(\phi(\tau, \sigma)\) with respect to \(\sigma\). Then the functions \(\phi_m(\tau)\) are quasi-modular forms of weight \(k + m\) for the group \(\Gamma\).
\end{proposition}

\begin{proof}
This is given by equation (6) on p. 31 of \cite{eichler_zagier}. Specifically, one notes that the transformation law presented is exactly the transformation law of a quasi-modular form.
\end{proof}

\begin{corollary}
Let \(\phi(\tau, \sigma)\) be a Jacobi form of weight \(k\) for some group \(\Gamma\). Then the function
\[
\Big(\frac{\partial}{\partial \sigma}\Big)^m \phi(\tau, \sigma)\Big|_{\sigma = 0}
\]
is quasi-modular of weight \(k + m\).
\end{corollary}

\begin{proof}
This is nothing but \(m!\) times the Taylor coefficient, by definition.
\end{proof}

The key note for us is that the functions introduced in Definition \ref{def_theta}, namely
\[
\vartheta_r(q, z) = \sum_{m \in \ZZ + r} q^{\frac{1}{2}m^2}z^m
\]
are Jacobi forms of weight \(k = \frac{1}{2}\) and index \(m = \frac{1}{2}\), for some subgroup \(\Gamma \subset SL_2(\ZZ)\) which depends on the rational number \(r\).

\begin{remark}
For the remainder of the paper, we will switch back to the notation \(\phi(q,z)\) for Jacobi forms/\(\vartheta\)-functions, where \(q = e^{2\pi i \tau}\) and \(z = e^{2\pi i \sigma}\), as it provides for cleaner notation.
\end{remark}

\section{Proofs of main Theorems}

We will now prove Theorem \ref{thm_main_A}.

\begin{proof}[Proof of Theorem \ref{thm_main_A}]
The proof of this will closely follow the proof of Theorem \ref{thm_old} (in \cite{andrews_rose}). We will begin by assuming for simplicity that \(n \notin S\).

We will begin by noting a version of the Jacobi triple product as it pertains to the functions \(\vartheta_r(q, z)\), which is
\begin{equation}\label{eq_JTP1}
\sum_{m \in \ZZ + r} q^{\frac{1}{2}m^2}z^m = z^{r}q^{\frac{r^2}{2}}\prod_{m > 0} (1 - q^m)(1 + zq^{m + r - \frac{1}{2}})(1 + z^{-1}q^{m - r - \frac{1}{2}})
\end{equation}
if \(r \neq \frac{1}{2}\), and
\begin{equation}\label{eq_JTP2}
\sum_{m \in \ZZ + \frac{1}{2}} q^{\frac{1}{2}m^2}z^m = (z^{1/2} + z^{-1/2})q^{\frac{1}{8}}\prod_{m > 0} (1 - q^m)(1 + zq^m)(1 + z^{-1}q^m)
\end{equation}
if \(r = \frac{1}{2}\).

Let us now note that
\begin{align*}
\sum_{k=0}^\infty A_{S,n,k}(q) x^{2k} &= \prod_{\ell \in S} \prod_{m = 0}^\infty \Big( 1 - \frac{x^2q^{nm + \ell}}{(1 - q^{nm + \ell})^2}\Big) \\
&= \prod_{\ell \in S} \prod_{m=0}^\infty \frac{1 - (x^2 + 2)q^{nm + \ell} + q^{2(nm + \ell)}}{(1 - q^{nm + \ell})^2}.
\end{align*}
If we then write \(x = z^{1/2} - z^{-1/2}\) (and hence \(x^2 + 2 = z + z^{-1}\)) we find that this is equal to
\[
\prod_{\ell \in S} \prod_{m = 0}^\infty \frac{(1 - zq^{nm + \ell})(1 - z^{-1}q^{nm + \ell})}{(1 - q^{nm + \ell})^2}.
\]
Due to the symmetry of \(S\), for each term \((1 - z^{\pm1}q^{nm + \ell})\) there is a corresponding term \((1 - z^{\mp 1}q^{nm + n - \ell})\) (similarly for the denominator), and so we can write this as
\[
\prod_{\ell \in S} \prod_{m=0}^\infty \frac{(1 - zq^{nm + \ell})(1 - z^{-1}q^{n(m+1) - \ell})}{(1 - q^{nm + \ell})(1 - q^{n(m+1) - \ell})}
=
\prod_{\ell \in S} \prod_{m=0}^\infty \frac{(1 - q^{nm})(1 - zq^{nm + \ell})(1 - z^{-1}q^{n(m+1) - \ell})}{(1 - q^{nm})(1 - q^{nm + \ell})(1 - q^{n(m+1) - \ell})}.
\]

We now apply the equation \eqref{eq_JTP1} with \(r = \alpha(n, \ell)\) to the numerator and denominator from which we obtain
\[
\prod_{\ell \in S} \frac{z^{\alpha(n, \ell)} \vartheta_{\alpha(n, \ell)}(q^n, -z)}{\vartheta_{\alpha(n, \ell)}(q^n,-1)}.
\]
Since \(S\) is symmetric and \(\alpha(n, \ell) = -\alpha(n, n - \ell)\), it follows that this is
\[
\prod_{\ell \in S} \frac{\vartheta_{\alpha(n, \ell)}(q^n, -z)}{\vartheta_{\alpha(n, \ell)}(q^n)}
\]
as claimed.

In the case that \(n \in S\), we will have an extra term in the product of the form
\begin{align*}
\prod_{m=1}^\infty \Big(1 - \frac{x^2q^{nm}}{(1 - q^{nm})^2}\Big)
&=
\prod_{m=1}^\infty \frac{(1 - zq^{nm})(1 - z^{-1}q^{nm})}{(1 - q^{nm})^2} \\
&= \prod_{m=1}^\infty \frac{(1 - q^{nm})(1 - zq^{nm})(1 - z^{-1}q^{nm})}{(1 - q^{nm})^3}
\end{align*}
which by equation \eqref{eq_JTP2} is simply
\begin{align*}
\frac{1}{\big((-z)^{1/2} + (-z)^{-1/2}\big)\eta(q^n)^3}\sum_{m \in \ZZ + \tfrac{1}{2}} q^{\frac{1}{2}nm^2}(-z)^m &= \frac{-i}{(z^{1/2} - z^{-1/2}) \eta(q^n)^3 }\vartheta_{1/2}(q^n, -z)
\end{align*}
as claimed.

\end{proof}

The proof of Theorem \ref{thm_main_B} is obtained {\em mutatis mutandi}, if we note that from equation \eqref{eq_JTP2} the term \(z^{1/2} + z^{-1/2} = \sqrt{x^2 + 4}\) when \(x = z^{1/2} - z^{-1/2}\).

\subsection{Modularity of the generating functions}

To prove modularity, we note that the left-hand side of the expression \eqref{eq_thm_main_A} is a product of Jacobi \(\vartheta\)-functions, and in particular it is a Jacobi form of weight 0 and index \(s/2\).

We need one quick proposition before we can properly discuss modularity.

\begin{proposition}\label{prop_even_odd}
Fix \(n\) and a symmetric set \(S\). The Jacobi form \(\prod_{\ell \in S} \vartheta_{\alpha(n, \ell)}(q^n, -z)\) is an even function of \(\sigma\) if and only if \(n \notin S\), and it is an odd function of \(\sigma\) if and only if \(n \in S\) (where \(z = e^{2\pi i \sigma}\)).
\end{proposition}

\begin{proof}
Note that a function of \(\sigma\) being even is equivalent to it being invariant under the exchange of \(z \leftrightarrow z^{-1}\) (and similarly for it being odd).

The key here is the symmetry requirement. We essentially have three cases: \(\ell = n, \ell = \frac{n}{2}\), and other \(\ell\). We will only prove the first of these, the rest of them being similar.

If \(\ell = n\), we have the function \(\vartheta_{1/2}(q, -z)\) as part of the product. If we swap \(z\) for \(z^{-1}\) in this function, then we obtain
\begin{align*}
\sum_{m \in \ZZ + \frac{1}{2}} q^{\frac{1}{2}m^2}(-z^{-1})^m
&=
\sum_{m \in \ZZ + \frac{1}{2}} (-1)^m q^{\frac{1}{2}m^2}z^{-m} \\
&=
\sum_{m \in \ZZ + \frac{1}{2}} (-1)^{-m} q^{\frac{1}{2}m^2}z^{m}
\end{align*}
Since \(-m \equiv m + 1 \pmod 2\) for \(m \in \ZZ + \tfrac{1}{2}\), the claim follows.
\end{proof}

As usual, let us assume that \(n \notin S\). In particular, let us write it (by Propositions \ref{prop_jac_qmod} and \ref{prop_even_odd}) as
\[
\sum_{k=0}^\infty (-1)^k A_{S, n, k}(q) x^{2k}  =   \sum_{w=0}^\infty \widetilde{A}_{S, n, w}(q) \sigma^{2w}
\]
(if \(n \in S\), then we just increase the exponents of \(x\) and \(\sigma\) by one). It follows then that the functions \(\widetilde{A}_{S, n, w}(q)\) are quasi-modular of pure weight \(2w\) each. Our goal is now to find explicit expressions for \(\widetilde{A}_{S, n, w}(q)\) and to relate them to \(A_{S, n, k}(q)\).

The first thing that we note is that, since \(x = z^{1/2} - z^{-1/2}\) and \(z = e^{2\pi i \sigma}\) that we can write
\[
x = 2i \sin (\pi \sigma) \qquad \text{and} \qquad \sigma = \frac{1}{\pi} \arcsin(-ix/2)
\]
to translate between these two expressions. In particular, (letting \(x \mapsto ix\)) if we write this as
\[
\sum_{k=0}^\infty A_{S, n, k}(q) x^{2k}  =   \sum_{w=0}^\infty \widetilde{A}_{S, n, w}(q) \Big( \frac{1}{\pi} \arcsin(x/2)\Big)^{2w}
\]
then we see by looking at the coefficient of \(x^{2k}\) on the right-hand side that the terms of pure weight \(2w\)  of \(A_{S, n, k}(q)\) are simply \(\widetilde{A}_{S, n, w}(q)\) and that, moreover, the coefficients are given by
\[
 [x^{2k}]\sum_{w=0}^\infty\frac{1}{\pi^{2w}} \arcsin(x/2)^{2w}
\]
which is a finite sum, since \(\arcsin(x/2) = x/2 + O(x^2)\) and hence for \(w > k\), we have that \([x^{2k}]\arcsin(x/2)^{2w} = 0\).

Let us now prove Theorem \ref{thm_top_weight}.

\begin{proof}[Proof of Theorem \ref{thm_top_weight}]
Let us as before assume first that \(n \notin S\). We have then that
\[
\sum_{k=0}^\infty (-1)^k A_{S, n, k}(q) x^{2k} = \prod_{j=1}^s \frac{\vartheta_{\alpha(n, \ell_j)}(q^n, -z)}{\vartheta_{\alpha(n, \ell_j)}(q^n)} = \sum_{k=0}^\infty \widetilde{A}_{S,n,k}(q) \sigma^{2k}
\]
and we seek to understand the functions \(\widetilde{A}_{S,n,k}(q)\). These are given by
\begin{align*}
\Big(\frac{\partial}{\partial \sigma}\Big)^{2k} \prod_{j=1}^s \frac{\vartheta_{\alpha(n, \ell_j)}(q^n, -z)}{\vartheta_{\alpha(n, \ell_j)}(q^n)} \Big|_{\sigma = 0}
&=
\big(2\pi i D_z\big)^{2k} \prod_{j=1}^s \frac{\vartheta_{\alpha(n, \ell_j)}(q^n, -z)}{\vartheta_{\alpha(n, \ell_j)}(q^n)} \Big|_{z=1} \\
&=
(2\pi i)^{2k} D_z^{2k} \prod_{j=1}^s \frac{\vartheta_{\alpha(n, \ell_j)}(q^n, -z)}{\vartheta_{\alpha(n, \ell_j)}(q^n)} \Big|_{z=1}.
\end{align*}

Since
\[
\Big(\frac{d}{dx}\Big)^k \prod_{j=1}^s f_j(x) = \sum_{i_1 + \cdots + i_s = k} {k \choose i_1, \ldots, i_s} \prod_{j=1}^s f_j^{(i_j)}(x)
\]
it follows that that the pure weight terms \(\widetilde{A}_{S, n, k}(q)\) are, up to scaling,
\[
D_z^{2k} \prod_{j=1}^s \frac{\vartheta_{\alpha(n, \ell_j)}(q^n, -z)}{\vartheta_{\alpha(n, \ell_j)}(q^n)} \Big|_{z=1}
=
\sum_{i_1 + \cdots + i_s = 2k} {2k \choose i_1, \ldots, i_s} \prod_{j=1}^s \frac{\vartheta_{\alpha(n, \ell_j)}^{(i_j)}(q^n)}{\vartheta_{\alpha(n, \ell_j)}(q^n)}
\]
as claimed.

In the  case that \(n \in S\), the difference is that \(\vartheta_{1/2}^{(2k)}(q^n,-1) = 0\) (since \(\vartheta_{1/2}(q, z)\) is odd due to Proposition \ref{prop_even_odd}) for all \(k \geq 0\). Since \(\vartheta(q, z)\) satisfies the heat equation, we find that
\begin{align*}
D_z^{2k + 1}\vartheta(q^n, -z)
&= D_z^{2k}D_z\vartheta(q^n, -z) \\
&= \Big(\frac{2}{n}\Big)^{k}D_q^k D_z\vartheta(q^n, -z).
\end{align*}
Since we know classically that \(\eta(q)^3 = \sum_{m \in \ZZ}(-1)^m (m + \tfrac{1}{2})q^{\frac{1}{2}(m + \frac{1}{2})^2}\), it follows that
\[
D_z\vartheta(q^n, -z)\Big|_{z=1} = i\eta(q^n)^3
\]
and so
\[
D_z^{2k + 1}\vartheta(q^n, -z) = i\Big(\frac{2}{n}\Big)^k D_q^k \eta(q^n)^3.
\]

We can put all of this together now to find that---again, up to scaling---if \(n \in S\) that the pure-weight term \(\widetilde{A}_{S, n, k}(q)\) is given by
\begin{multline*}
D_z^{2k+1}\bigg[-i  \frac{\prod_{j=0}^s \vartheta_{\alpha(n, \ell_j)}(q^n, -z)}{\eta(q^n)^3 \prod_{j=1}^s \vartheta_{\alpha(n, \ell_j)}(q^n)}\bigg]
\\= 
\sum_{(2i_0+1) + i_1 + \cdots + i_s = 2k + 1} {2k+1 \choose 2i_0 + 1, i_1, \ldots, i_s} \Big(\frac{2}{n}\Big)^{i_0} \frac{\big[D_q^{i_0} \eta(q^n)^3\big] \prod_{j=1}^s \vartheta_{\alpha(n, \ell_j)}^{(i_j)}(q^n)}{\eta(q^n)^3\prod_{j=1}^s \vartheta_{\alpha(n, \ell_j)}(q^n)}
\end{multline*}
\end{proof}

\subsection{The functions \(B_{S, n, k}(q)\)}

It seems natural to look for a similar expression for the functions \(B_{S, n, k}(q)\). As for the case of the functions \(A_{S, n, k}(q)\), the expression that we obtain depends on whether or not \(n \in S\). For the case that \(n \notin S\), an argument identical to that above yields that \(B_{S, n, k}(q)\) is quasi-modular of impure weight, with the pure weight terms being given by multiples of
\[
\sum_{i_1 + \cdots + i_s = 2w} {2w \choose i_1, \ldots, i_s} \prod_{j=1}^s \frac{\widetilde{\vartheta}_{\alpha(n, \ell_j)}^{(i_j)}(q^n)}{\widetilde{\vartheta}_{\alpha(n, \ell_j)}(q^n)}
\]

In the case that \(n \in S\), it is a bit more complicated due to the the presence of the \(\sqrt{1 + \big(\tfrac{x}{2}\big)^2}\). However, by expanding this out in powers of \(x\) we obtain that
\[
\sqrt{1 + \big(\tfrac{x}{2}\big)^2}\sum_{k=0}^\infty B_{S, n, k}(q)x^{2k} = \sum_{k=0}^\infty \Big(\sum_{m=0}^k a_m B_{S, n, k-m}(q)\Big) x^{2k}
\]
(where \(\sqrt{1 + \big(\tfrac{x}{2}\big)^2} = \sum_{k=0}^\infty a_k x^{2k}\)), and so we can recursively determine that these functions are themselves quasi-modular, although the expression is more complicated.

\section{Further and related work}

The generating functions \(A_k(q), C_k(q)\) studied in \cite{andrews_rose} arise naturally in the following problem in enumerative geometry.

Let \(A\) be a polarized abelian surface with polarization \(L\) of type \((1, n)\). There is a \((g - 2)\)-dimensional family of genus \(g\) curves in the homology class \(\beta = c_1(L)^\vee\) up to translation, which is the codimension of the hyperelliptic locus in the moduli space of genus \(g\) curves. We would then naturally expect that there is a finite number of genus \(g\) hyperelliptic curves in this class.

Let \(F_g(q)\) denote the generating function for counting such curves. In \cite{rose_paper} it is shown (assuming the crepant resolution conjecture) that we can write \(F_g(q)\) as a polynomial of degree \((g-1)\) in the functions \(A_k(q^4), C_k(q^2)\) (where we assign the weight \(k\) to each of these functions), and hence \(F_g(q)\) is a quasi-modular form for \(\Gamma_0(4)\). Interestingly enough, there is strong numerical evidence that it is in fact quasi-modular for \(SL_2(\ZZ)\), due to some remarkable cancellations, a fact that will be further investigated.

Furthermore, these generating functions can be described in terms of Hurwitz-Hodge type integrals over certain moduli spaces of Hurwitz covers, more detail of which will appear in forthcoming work by the author.

A future goal is to then understand whether or not these more general functions can be found in an enumerative context.

\begin{appendices}
\section{Some coefficients}
For the sake of completeness, we will include a few coefficients for some small values of \(n\) and symmetric sets \(S\). For each of the given values of \(n, S\), we will write
\[
A_{S, n, k}(q) = \sum_{m=1}^\infty a_{S, n, k, m} q^m.
\]
All of the following have been computed with a custom Sage program.

\begin{enumerate}

\item \(n=3, S = \{1,2\}\)

\begin{tabular}{c|cccccccccccccccc}
$(k,m)$ & 1 & 2 & 3 & 4 & 5 & 6 & 7 & 8 & 9 & 10 & 11 & 12 & 13 & 14 & 15 \\
\hline
1 & 1 & 3 & 3 & 7 & 6 & 9 & 8 & 15 & 9 & 18 & 12 & 21 & 14 & 24 & 18 \\
2 & & & 1 & 2 & 6 & 12 & 20 & 30 & 48 & 66 & 90 & 124 & 154 & 204 & 240\\
3 & & & & & & & 1 & 3 & 7 & 15 & 30 & 49 & 87 & 132 & 210\\
4 & & & & & & & & & & & & 1 & 2 & 6 & 12
\end{tabular}

\item \(n=4, S = \{1,3\}\)

\begin{tabular}{c|cccccccccccccccc}
$(k, m)$ & 1 & 2 & 3 & 4 & 5 & 6 & 7 & 8 & 9 & 10 & 11 & 12 & 13 & 14 & 15 & 16\\
\hline
1 & 1 & 2 & 4 & 4 & 6 & 8 & 8 & 8 & 13 & 12 & 12 & 16 & 14 & 16 & 24 & 16\\
2 & & & & 1 & 2 & 4 & 8 & 14 & 18 & 28 & 40 & 52 & 70 & 88 & 104 & 140\\
3 & & & & & & & & & 1 & 2 & 4 & 8 & 14 & 24 & 40 & 56\\
4 & & & & & & & & & & & & & & & & 1
\end{tabular}

\item \(n=5, S = \{1,4\}\)

\begin{tabular}{c|cccccccccccccccc}
$(k, m)$ & 1 & 2 & 3 & 4 & 5 & 6 & 7 & 8 & 9 & 10 & 11 & 12 & 13 & 14 & 15 & 16\\
\hline
1 & 1 & 2 & 3 & 5 & 5 & 7 & 7 & 10 & 10 & 10 &12 & 17 & 13 & 15 & 15 & 21\\
2 & & & & & 1 & 2 & 4 & 6 & 10 & 16 & 20 & 26 & 38 &50 & 62 & 74\\
3 & & & & & & & & & & & 1 & 2 & 3 & 5 & 9 & 15
\end{tabular}

\item \(n=5, S = \{2,3\}\)

\begin{tabular}{c|cccccccccccccccc}
$(k, m)$ & 1 & 2 & 3 & 4 & 5 & 6 & 7 & 8 & 9 & 10 & 11 & 12 & 13 & 14 & 15 & 16\\
\hline
1 & & 1 & 1 & 2 & 0 & 5 & 1 & 5 & 3 & 5 & 0&11 & 1 & 9 & 5 & 10\\
2 & & & & & 1 & 0 & 2 & 2 & 4 & 6 & 10 & 8 & 16 & 18 & 22 & 30\\
3 & & & & & & & & & & & & 1 & 1 & 2 & 4 & 5

\end{tabular}

\item \(n=6, S = \{1,5\}\)

\begin{tabular}{c|cccccccccccccccc}
$(k, m)$ & 1 & 2 & 3 & 4 & 5 & 6 & 7 & 8 & 9 & 10 & 11 & 12 & 13 & 14 & 15 & 16\\
\hline
1 & 1& 2& 3& 4& 6& 6& 8& 8& 9& 12& 12& 12& 14& 16& 18& 16\\
2 & & & & & & 1& 2& 4& 6& 8& 12& 18& 22& 28& 36& 48\\
3 & && & & & & & & & & & & 1& 2& 3& 4

\end{tabular}

\end{enumerate}
\end{appendices}

\bibliographystyle{alpha}
\bibliography{cheb}

\end{document}